\documentclass[12pt]{amsart}
\usepackage{amsmath,amssymb,amsbsy,amsfonts,amsthm,latexsym,mathabx,
            amsopn,amstext,amsxtra,euscript,amscd,stmaryrd,mathrsfs,
            cite,array,mathtools,enumerate}

\usepackage{url}
\usepackage[colorlinks,linkcolor=blue,anchorcolor=blue,citecolor=blue,backref=page]{hyperref}

\usepackage[norefs,nocites]{refcheck}
\usepackage{color}
\usepackage{float}

\hypersetup{breaklinks=true}

\usepackage[english]{babel}

\usepackage{mathtools}
\usepackage{todonotes}

\usepackage{enumitem}

\usepackage{mathtools}
\usepackage{todonotes}
\usepackage{url}
\usepackage[colorlinks,linkcolor=blue,anchorcolor=blue,citecolor=blue,backref=page]{hyperref}

\begin{document}

\newtheorem{theorem}{Theorem}
\newtheorem{lemma}[theorem]{Lemma}
\newtheorem{claim}[theorem]{Claim}
\newtheorem{cor}[theorem]{Corollary}
\newtheorem{prop}[theorem]{Proposition}
\newtheorem{definition}{Definition}
\newtheorem{question}[theorem]{Open Question}
\newtheorem{example}[theorem]{Example}
\newtheorem{remark}[theorem]{Remark}

\numberwithin{equation}{section}
\numberwithin{theorem}{section}

 \newcommand{\F}{\mathbb{F}}
\newcommand{\K}{\mathbb{K}}
\newcommand{\D}[1]{D\(#1\)}
\def\scr{\scriptstyle}
\def\\{\cr}
\def\({\left(}
\def\){\right)}
\def\<{\langle}
\def\>{\rangle}
\def\fl#1{\left\lfloor#1\right\rfloor}
\def\rf#1{\left\lceil#1\right\rceil}
\def\le{\leqslant}
\def\ge{\geqslant}
\def\eps{\varepsilon}
\def\mand{\qquad\mbox{and}\qquad}

\def\vec#1{\mathbf{#1}}
\def\ve {\vec{e}}
\def\vh {\vec{h}}
\def\vG {\vec{G}}
\def\vQ {\vec{Q}}

\newcommand{\discr}{\operatorname{discr}}
\newcommand{\wdeg}{\operatorname{wdeg}}

\newcommand{\lcm}{\operatorname{lcm}}

\def\bl#1{\begin{color}{blue}#1\end{color}} 

\newcommand{\C}{\mathbb{C}}
\newcommand{\Fq}{\mathbb{F}_q}
\newcommand{\Fp}{\mathbb{F}_p}
\newcommand{\Disc}[1]{\mathrm{Disc}\(#1\)}
\newcommand{\Res}[1]{\mathrm{Res}\(#1\)}
\newcommand{\ord}{\mathrm{ord}}

\newcommand{\Q}{\mathbb{Q}}
\renewcommand{\L}{\mathbb{L}}
\renewcommand{\L}{\mathbb{L}}
\renewcommand{\P}{\mathbb{P}}

\newcommand{\Norm}{\mathrm{Norm}}

\def\cA{{\mathcal A}}
\def\cB{{\mathcal B}}
\def\cC{{\mathcal C}}
\def\cD{{\mathcal D}}
\def\cE{{\mathcal E}}
\def\cF{{\mathcal F}}
\def\cG{{\mathcal G}}
\def\cH{{\mathcal H}}
\def\cI{{\mathcal I}}
\def\cJ{{\mathcal J}}
\def\cK{{\mathcal K}}
\def\cL{{\mathcal L}}
\def\cM{{\mathcal M}}
\def\cN{{\mathcal N}}
\def\cO{{\mathcal O}}
\def\cP{{\mathcal P}}
\def\cQ{{\mathcal Q}}
\def\cR{{\mathcal R}}
\def\cS{{\mathcal S}}
\def\cT{{\mathcal T}}
\def\cU{{\mathcal U}}
\def\cV{{\mathcal V}}
\def\cW{{\mathcal W}}
\def\cX{{\mathcal X}}
\def\cY{{\mathcal Y}}
\def\cZ{{\mathcal Z}}

\def\fra{{\mathfrak a}} 
\def\frb{{\mathfrak b}}
\def\frc{{\mathfrak c}}
\def\frd{{\mathfrak d}}
\def\fre{{\mathfrak e}}
\def\frf{{\mathfrak f}}
\def\frg{{\mathfrak g}}
\def\frh{{\mathfrak h}}
\def\fri{{\mathfrak i}}
\def\frj{{\mathfrak j}}
\def\frk{{\mathfrak k}}
\def\frl{{\mathfrak l}}
\def\frm{{\mathfrak m}}
\def\frn{{\mathfrak n}}
\def\fro{{\mathfrak o}}
\def\frp{{\mathfrak p}}
\def\frq{{\mathfrak q}}
\def\frr{{\mathfrak r}}
\def\frs{{\mathfrak s}}
\def\frt{{\mathfrak t}}
\def\fru{{\mathfrak u}}
\def\frv{{\mathfrak v}}
\def\frw{{\mathfrak w}}
\def\frx{{\mathfrak x}}
\def\fry{{\mathfrak y}}
\def\frz{{\mathfrak z}}

\def\ov#1{\overline{#1}}
\def \brho{\boldsymbol{\rho}}

\def \fB {\mathfrak B}
\def \fG {\mathfrak G}
\def \fP {\mathfrak P}

\def \Prob{{\mathrm {}}}
\def\e{\mathbf{e}}
\def\ep{{\mathbf{\,e}}_p}
\def\epp{{\mathbf{\,e}}_{p^2}}
\def\em{{\mathbf{\,e}}_m}

\newcommand{\sR}{\ensuremath{\mathscr{R}}}
\newcommand{\sDI}{\ensuremath{\mathscr{DI}}}
\newcommand{\DI}{\ensuremath{\mathrm{DI}}}

\newcommand{\Orb}[1]{\mathrm{Orb}\(#1\)}
\newcommand{\aOrb}[1]{\overline{\mathrm{Orb}}\(#1\)}
\def \PrePer{{\mathrm{PrePer}}}
\def \Per{{\mathrm{Per}}}

\def \Nm{{\mathrm{Nm}}}
\def \Gal{{\mathrm{Gal}}}

\newenvironment{notation}[0]{%
  \begin{list}%
    {}%
    {\setlength{\itemindent}{0pt}
     \setlength{\labelwidth}{1\parindent}
     \setlength{\labelsep}{\parindent}
     \setlength{\leftmargin}{2\parindent}
     \setlength{\itemsep}{0pt}
     }%
   }%
  {\end{list}}

\definecolor{dgreen}{rgb}{0.,0.6,0.}
\def\tgreen#1{\begin{color}{dgreen}{\it{#1}}\end{color}}
\def\tblue#1{\begin{color}{blue}{\it{#1}}\end{color}}
\def\tred#1{\begin{color}{red}#1\end{color}}
\def\tmagenta#1{\begin{color}{magenta}{\it{#1}}\end{color}}
\def\tNavyBlue#1{\begin{color}{NavyBlue}{\it{#1}}\end{color}}
\def\tMaroon#1{\begin{color}{Maroon}{\it{#1}}\end{color}}

\title[Irreducible polynomials over thin subgroups]{Counting Irreducible polynomials with coefficients from thin subgroups}

\author[A. Ostafe] {Alina Ostafe}
\address{School of Mathematics and Statistics, University of New South Wales, Sydney NSW 2052, Australia}
\email{alina.ostafe@unsw.edu.au}

 \author[I.~E.~Shparlinski]{Igor E. Shparlinski}
 \address{I.E.S.: School of Mathematics and Statistics, University of New South Wales.
 Sydney, NSW 2052, Australia}
 \email{igor.shparlinski@unsw.edu.au}
 
 \begin{abstract} L.~Bary-Soroker and R.~Shmueli (2026) have given an asymptotic formula for the number 
 of irreducible polynomials over the finite fields $\F_q$ of $q$ elements, such that their coefficients 
 are perfect squares in $\F_q$ and also extended this to classes of polynomials with coefficients 
 described by finitely many unions of intersections of polynomial images. Here we use a different approach, 
 which allows us to obtain another generalisation of this result to polynomials with coefficients from small
 subgroups of $\F_q^*$. As a demonstration of the power of our approach, we also use it to count such irreducible 
 polynomials with an additional condition, namely, with a prescribed value  of their discriminant. This generalisation 
 seems to be unachievable via the approach of L.~Bary-Soroker and R.~Shmueli (2026). \end{abstract}

\pagenumbering{arabic}


\maketitle

\tableofcontents

\section{Introduction}
\subsection{Motivation}
\label{sec: mot} 

Let $\F_q$ be a finite field of $q$ elements, and let 
$\vG = (\cG_0,    \ldots, \cG_{n-1})$ 
be a sequence of $n$  arbitrary multiplicative subgroups of $\F_q^*$, and let 
$\vh = (h_0,  \ldots, h_{n-1})\in \(\F_q^*\)^n$ 
be a sequence of $n$ arbitrary elements of $\F_q^*$. 
We view  $\vG$ as a subgroup of $\(\F_q^*\)^n$ and thus write $\vG \le \(\F_q^*\)^n$ 
to indicate this. 

We denote by $\cI_n(\vh, \vG)$ the set of monic irreducible polynomials $f\in \F_q[X]$
with coefficients from co-sets $h_i\cG_i$, $i = 0, \ldots, n$, that is, 
of the form
\[
f(X) = X^n + a_{n-1} X^{n-1} + \ldots + a_0, \qquad a_i \in h_i\cG_i, \ i = 0, \ldots, n-1.
\]

When $q$ is odd, $\vG =\vQ =\( \cQ, \ldots,  \cQ \)$, where $\cQ = \{u^2:~u \in \F_q^*\}$  is the subgroup
of squares, and $\vh = {\mathbf 1} = (1,\ldots,1)$ is the identity vector, Bary-Soroker and Shmueli~\cite[Theorem~1]{B-SShm} 
have established the asymptotic formula 
\begin{equation}
\label{eq:square coeffs}
\# \cI_n({\mathbf 1} , \vQ) = \frac{q^n}{n2^n} + O_n\(q^{n-1/2}\), 
\end{equation}
where $O_n$ indicates that the implied constant depends only on $n$. 

In fact, the asymptotic formula~\eqref{eq:square coeffs}
is derived from a more general result~\cite[Theorem~2]{B-SShm}  on counting irreducible polynomials 
with coefficients from sets defined as finite unions and intersections of polynomial images. 
In particular, if $\vG \le  \(\F_q^*\)^n$ is of bounded index
\[
\left[ \(\F_q^*\)^n: \vG\right] \le M 
\]
for some $M>0$,  then 
\begin{equation}
\label{eq:subgr coeffs}
\# \cI_n(\vh, \vG) = \frac{1}{n} \prod_{j=0}^{n-1} \# \cG_i+ O_{n,M}\(q^{n-1/2}\), 
\end{equation}
where the implied constant in $O_{n,M}$ depends only on $n$ and $M$. 

\subsection{Our results}
\label{sec: res} 
The approach of Bary-Soroker and Shmueli~\cite{B-SShm} relies on a version of the 
{\it Chebotarev Density Theorem\/},  and thus leads to constants, which,  while effective, are sometimes
not easy to explicitly estimate, and these estimates are rather poor (especially those depending on 
the parameter $M$ in~\eqref{eq:subgr coeffs}). 

Here we use a different approach, which allows us to get a nontrivial asymptotic formula 
for subgroups of  index which grows as a power of $q$. Furthermore, to show the power of 
our approach, we refine the question and count polynomials from $ \cI_n(\vh, \vG)$ 
with a given discriminant. Thus, for $d \in \F_q^*$, we define the set
\[
\cI_{d,n} (\vh, \vG)= \{f \in  \cI_n(\vh, \vG):~\discr f = d\},
 \]
 where $\discr f$ denotes the discriminant of the polynomial $f$. 
 However our method requires some mild restriction on the characteristic of $\F_q$. 
 
Let  $\cI_{d,n} =  \cI_{d,n} \({\mathbf 1}, (\F_q^*)^n\)$ be the set of all monic irreducible  polynomials 
of degree $n$ over $\F_q$ with non-zero coefficients and the discriminant $d$, where we set ${\mathbf 1}=(1,\ldots,1)\in \F_q^n$.

 \begin{theorem}
\label{thm:IdhG} Let $\F_q$ be of characteristic $p > n\ge 2$. 
For $d \in  \F_q^*$, $\vh  \in \F_q^*$ and $\vG \le  \(\F_q^*\)^n$ we have 
\begin{align*}
\left| \#\cI_{d,n} (\vh, \vG) -  \frac{\# \cI_{d,n}}{(q-1)^n} \prod_{j=0}^{n-1} \# \cG_i \right| \le &
 \frac{n(n-1)}{2}  \# \cI_{d,n}  q^{-1/2} \\
 &\qquad\quad+  \frac{n^2(n^2-1)(3n+2)}{24} q^{n-2}.
\end{align*}
\end{theorem}

Summing over all $d \in \F_q^*$, we immediately 
obtain the following uniform version of~\eqref{eq:subgr coeffs}. 

 \begin{cor}
\label{cor:IhG} Let $\F_q$ be of characteristic $p > n\ge 2$. 
For  $\vh  \in \F_q^*$ and $\vG \le  \(\F_q^*\)^n$, we have 
\[
 \#\cI_n(\vh, \vG)=  \frac{1}{n} \prod_{j=0}^{n-1} \# \cG_i + O\(n^2q^{n-1/2}+n^5 q^{n-1}\),
\]
where the implied constant is absolute. 
\end{cor}

In particular, for a fixed $n$,  the result  of   Corollary~\ref{cor:IhG} is nontrivial and actually gives an asymptotic 
formula with a power saving in the error term for subgroups $\vG$ of 
index
\[
\left[ \(\F_q^*\)^n: \vG\right] \le q^{1/2-\varepsilon} 
\]
with any fixed $\varepsilon > 0$.

While this is not the main focus of this paper, getting an asymptotic formula for $\# \cI_{d,n}$ is certainly a natural question, which 
we address in Appendix~\ref{app:Irr Poly Discr}. The celebrated Stickelberger theorem~\cite{St} (see also~\cite{Dalen,Skol}),  
suggests that 
$\#\cI_{d,n} = \(2/n+ o_n(1)\)  q^{n-1}$ for an odd $q$,  provided that $\chi_2(d) = (-1)^{n-1}$, 
where $\chi_2$ is the quadratic character of $\F_q$,  and $\#\cI_{d,n} = 0$, otherwise. Theorem~\ref{thm:Irr Discr} below 
confirms this.

\section{Algebraic geometry tools}

\subsection{Rational points on varieties}

The following statement, which is well-known for hypersurfaces as the {\it Schwartz--Zippel Lemma\/}
has been extended to varieties by Bukh and Tsimerman~\cite[Lemma~14]{BuTs}, 
and further generalised by Xu~\cite[Lemma~1.7]{Xu} as follows.

\begin{lemma}
\label{lem: Gen S-Z_Lem} Let $\cV \subseteq \overline \F_q^n $ be an affine  variety over 
 $\overline \F_q$   of co-dimension  $r$ 
and of degree $D$. Then 
\[
\sharp \(\cV\cap  \F_q^n\) \le D q^{n-r}. 
\]  
\end{lemma} 

\subsection{Degrees of intersections}

We now need a version of   Bezout's Theorem, see~\cite[Corollary~2.5]{EisHar:3264}.  

\begin{lemma}
\label{lem: Deg Intersect} Let $\cV_1, \cV_2 \subseteq   \P^n(\ov\F_q) $ be projective varieties over 
$\overline \F_q$  of  degrees $D_1$ and $D_2$, respectively. Then $\cV_1 \cap   \cV_2$ 
 is of degree at most $D_1D_2$. 
\end{lemma}

\section{Discriminant variety}

\subsection{Irreducibility of the discriminant variety}

Let $\cD_{d,n}$ be the hypersurface of monic polynomials over $\overline \F_q$ (or more precisely 
of their coefficients) of degree $n$ and with a given discriminant $d \in \overline \F_q$, 
where $ \overline \F_q$ is the algebraic closure of $\F_q$. 

\begin{lemma}\label{lem:discr surf} 
 Let $\F_q$ be of odd characteristic. 
For any $d \in  \overline \F_q$  the hypersurface  $\cD_{d,n}$ is absolutely irreducible.  
\end{lemma}

\begin{proof}  
The hypersurface  $\cD_{d,n}$ is defined by the zero set of the polynomial
\[
F_d(A_0,\ldots,A_{n-1})=\discr\(X^n+A_{n-1}X^{n-1}+\cdots + A_0\) -d.
\]

We consider first $d=0$. Considering, if necessary, an extension of $\F_q$, we can 
assume that $q$ is large enough with respect to $n$. 

We suppose  that  we have the factorisation 
\begin{equation}
\label{eq:disc 0 fact}
F_0=\prod_{i=1}^s F_i^{\nu_i} 
\end{equation}
into $s$ distinct absolutely irreducible polynomials $F_i\in \F_q[A_0,\ldots,A_{n-1}]$ with
some multiplicities $\nu_1, \ldots, \nu_s \ge 1$. This factorisation together with the celebrated result 
of Lang and Weil~\cite{LaWe} implies that 
\begin{equation}
\label{eq:Count 1}
\# (\cD_{0,n}\cap \F_q[X]) = s q^{n-1} + O\(n^2q^{n^2-3/2}\)
\end{equation} 
with an absolute implied constant. 
On the other hand, by the classical result of Carlitz~\cite[Section~6]{Carl}, for any $n \ge 2$, we have 
\begin{equation}
\label{eq:Count 2}
\# \(\cD_{0,n}\cap \F_q[X]\)=q^{n-1}.
\end{equation} 
Comparing~\eqref{eq:Count 1} and~\eqref{eq:Count 2},  we see that $s=1$ in the factorisation~\eqref{eq:disc 0 fact}, and thus $F_0=F_1^{\nu_1}$. 

We now want to show that we must have $\nu_1=1$. For this it is enough to consider a specialisation of the coefficients in $F_0$. Indeed, let
\[
F_0(B,B_3,\ldots,B_n)=\discr\((X^2-B)\prod_{i=3}^n(X-B_i)\).
\]
Then simple computation shows that
\[
F_0(B,B_3,\ldots,B_n)=4B\prod_{3\le i<j\le n}(B_i-B_j)^2 \prod_{i=3}^n(B_i^2-B),
\] 
which cannot be a power of any polynomial (we recall $p$ is odd, and thus the constant factor of $F_0(B,B_3,\ldots,B_n)$ is nonzero). Therefore $F_0$ is absolutely irreducible.

We now consider $d\ne 0$.  

We define the {\it weighted degree\/} of $A_{i}$ as $\wdeg A_i =n- i$, and then extend 
this definition in a natural way  to first monomials in $A_0, \ldots, A_{n-1}$ and then to 
arbitrary polynomials, preserving its additivity
\[
\wdeg GH = \wdeg G + \wdeg H
\]
for any polynomials $G,H \in  \F_q[A_0,\ldots,A_{n-1}]$. 

From the explicit formula for the discriminant of a monic polynomial via its roots, we see that 
the discriminant is a symmetric homogeneous polynomial
of the roots $\alpha_1, \ldots, \alpha_n$  of degree $n(n-1)$. 
On the other hand, by our definition of the weighted degree, 
the elementary symmetric polynomials  $\sigma_i(\alpha_1, \ldots, \alpha_n)$ 
are of weighted degree  $i$ and also of degree $i$ as polynomials  
in $\alpha_1, \ldots, \alpha_n$.
Therefore, 
\begin{equation}
\label{eq:wdeg discr}
\wdeg \discr\(X^n+A_{n-1}X^{n-1}+\cdots + A_1 X+ A_0\) = n(n-1).
\end{equation}

Assume that 
\[
F_{d}=F_0-d=GH,
\]
where $G,H\in\F_q[A_0,\ldots,A_{n-1}]$ are of positive degrees and hence of 
positive weighted degrees $k$ and $\ell$, respectively. 
We see from~\eqref{eq:wdeg discr} that $k+\ell =  n(n-1)$ and moreover, 
\[
F_0=G^*H^*,
\]   
where $G^*$ is the top  homogeneous with respect to the weighted degree  part of $G$ and $H^*$ is similarly defined. Since 
\[
\wdeg G^* = \wdeg G\ge \deg G/n > 0
\]
and similarly for $H^*$, we conclude that $G^*$  and $H^*$ are not constant polynomials.
However, this contradicts the absolute irreducibility of $F_0$ proved above, which completes the proof.
\end{proof}

In particular, Lemma~\ref{lem:discr surf} implies that  $\cD_{d,n}$ is of dimension $\dim \cD_{d,n} = n-1$.

\subsection{Counting polynomials with non-relatively prime derivatives and 
given discriminant}

For $0 \le i< j < n$, let $\cR_{i,j,n}$ be the hypersurface of polynomials $f$  over $\overline \F_q$  
of degree $n$  defined by the equation $\Res{f^{(i)}, f^{(j)}}= 0$, where $\Res{f^{(i)}, f^{(j)}}$ is the resultant of the $i$th 
and $j$th derivatives of $f$. 

\begin{lemma}\label{lem:res surf}
Let $p>n$. For any  $0 \le i< j < n$  the hypersurface  $\cR_{i,j,n}$ is of dimension 
$\dim \cR_{i,j,n} = n-1$. 
\end{lemma}

\begin{proof}  Let $f(X)=X^n+A_{n-1}X^{n-1}+\cdots + A_0\in \F_q[A_0,\ldots,A_{n-1}]$. It is enough to show that for any  $0 \le i< j < n$, $\Res{f^{(i)}, f^{(j)}}$ as a polynomial in $ \F_q[A_0,\ldots,A_{n-1}]$ is not constant. Assume this is not the case, that is, $\Res{f^{(i)}, f^{(j)}}\in\F_q$ for some $0 \le i< j < n$. We can easily  obtain a contradiction by considering a specialisation of $A_0,\ldots,A_{n-1}$. Indeed, let $f(X)=X^n+AX^i \in \F_q[A]$. Then simple computation shows that
\[
\Res{f^{(i)}, f^{(j)}}=\pm \(n(n-1)\cdots (n-j+1)\)^{n-i}(i!)^{n-j}A^{n-j},
\]
which is a nonconstant polynomial (since $p>n$), contradicting our assumption above.
\end{proof}

Next, we consider intersections of the hypersurfaces  $\cD_{d,n}$ and  $\cR_{i,j,n}$. 

\begin{cor}\label{cor:discr res var}
Let $p>n$. For any $d \in \F_q$ and  $0 \le i< j < n$  the variety $\cD_{d,n} \cap \cR_{i,j,n}$ is of dimension 
at most $n-2$. 
\end{cor}

\begin{proof} 
 Since by Lemma~\ref{lem:res surf}  the polynomial defining $\cR_{i,j,n}$ is not identical to zero, 
 we see that  by Lemma~\ref{lem:discr surf} we have  $\dim \(\cD_{d,n} \cap  \cR_{i,j,n}\) < \dim \cD_{d,n}  = n-1$. Indeed, this follows from the absolute irreducibility of $\cD_{d,n}$ and the fact that $\deg \cD_{d,n}>\deg \cR_{i,j,n}$, and thus $\cD_{d,n}$ 
 and $\cR_{i,j,n}$ are relatively prime polynomials.
\end{proof}  

\section{Proof of Theorem~\ref{thm:IdhG}} 

\subsection{Splitting polynomials into  equivalence classes}

We say that two polynomials $f,g \in \F_q[X]$ are equivalent if for some $u \in \F_q$ we have 
$f(X) = g(X+u)$. 

We now show that for any nonzero  polynomial  $f \in \F_q[X]$, its equivalence class contains exactly 
$q$ polynomials. 
Indeed, since $f(X) = f(X+u)$ implies $f(X) = f(X+ju)$, $j = 0, \ldots, n$, we see that if $\alpha\in \overline \F_q$ 
is a root of $f$ then so are $\alpha - ju$, $j = 0, \ldots, n$, which are pairwise distinct for $p > n$, which 
 is impossible.

We also observe that the above equivalence relation preserves the 
irreducibility property and also the discriminant.

For   $f \in \cI_{n,d}$, let $\cJ_f$ be the equivalence class of $f$. Thus, 
we now split the set $\cI_{n,d}$   into $q^{-1} \#\cI_{n,d}$  equivalence classes  $\cJ_f$, 
which do not depend on the choice of the representative $f$.

We say that  $\cJ_f$ is {\it good\/} if for all integers $0 \le i < j \le n$ we have $\Res{f^{(i)}, f^{(j)}}\ne 0$.
Otherwise we say that  $\cJ_f$ is {\it bad\/}.

Denoting by $\cJ_f(\vh,\vG) = \cJ_f \cap  \cI_{d,n} (\vh, \vG)$, we conclude that 
\begin{equation}
\label{eq: Idn via Jf}
 \cI_{d,n} (\vh, \vG) = \frac{1}{q} \sum_{f \in \cI_{d,n}} \# \cJ_f(\vh,\vG). 
 \end{equation}

Next, we obtain an asymptotic 
formula for $\#\cJ_f(\vh,\vG)$ for good equivalence classes and then estimate the 
cardinality of the union of bad equivalence classes.

\subsection{Estimating the contribution from good and bad equivalence classes}

We denote $f_u(X) = f(X+u)$ and recall the Taylor formula
\[
f_u(X) = \sum_{i = 0}^n \frac{f^{(i)}(u)}{i!} X^i.
\]

Let $\cX_t$ denote the set of multiplicative characters of $\F_q^*$ of order dividing $t$, 
that is, with $\chi^t = \chi_0$, where $\chi_0$ is the principal character. 
It is also convenient to denote $\cX_t^* = \cX_t \setminus\{\chi_0\}$. 

We also denote $t_i = (q-1)/\#\cG_i$, $i = 0, \ldots, n-1$.

Using the orthogonality relation for characters (see, for example,~\cite[Chapter~3]{IwKow}), 
we see that for $f \in \cI_{d,n}$ we have 
\begin{equation}
\label{eq:char func}
\prod_{i=0}^{n-1} \frac{1}{t_i} \sum_{\psi_i \in \cX_i} \psi_i\(h_i^{-1} f^{(i)}(u)/i!\) = 
\begin{cases} 1, & \text{if } f_u \in  \cI_{d,n} (\vh, \vG), \\
0, & \text{otherwise}.
\end{cases} 
\end{equation}

Therefore, summing~\eqref{eq:char func} over all $u \in \F_q$ and separating the contributions
from the principal characters $\psi_i = \chi_0$, we see that for 
$f \in \cI_{d,n}$ we have
\begin{equation}
\label{eq:Jf and R}
\begin{split}
 \# \cJ_f(\vh,\vG)& =  \frac{1}{t_0\ldots t_{n-1}} \sum_{u \in \F_q} \prod_{i=0}^{n-1}  \(1 + 
 \sum_{\psi_i \in \cX_i^*} \psi_i\(h_i^{-1} f^{(i)}(u)/i!\)\) \\
& =  \frac{1}{t_0\ldots t_{n-1}}\(q  + R\),
\end{split}
\end{equation}
where 
\[
R=  \sum_{\substack{\cK \subseteq \{0, \ldots, n-1\}\\ \cK \ne\varnothing}}
 \sum_{\substack{\psi_i \in \cX_i^*\\ i \in \cK}}
   \sum_{u \in \F_q} \prod_{i\in \cK}   \psi_i\(h_i^{-1} f^{(i)}(u)/i!\).
\]

Clearly for a good  equivalence class, the polynomials $ f^{(i)}(X)$, $i = 0, \ldots, n-1$,  have no multiple or common roots. 
Hence, the Weil bound (see~\cite[Theorem~11.23]{IwKow}) applies to each of $t_0\ldots t_{n-1}-1$ sums over $u$ and yields 
\[
\left| \sum_{u \in \F_q} \prod_{i \in \cK}    \psi_i\(h_i^{-1} f^{(i)}(u)/i!\) \right| \le \frac{n(n-1)}{2} q^{1/2}. 
\]
Therefore 
\[
|R| \le \frac{n(n-1)}{2}   \(t_0\ldots t_{n-1}-1\)  q^{1/2},
\]
and substituting this in~\eqref{eq:Jf and R}, we see that 
\[
\left| \# \cJ_f(\vh,\vG) -   \frac{q}{t_0\ldots t_{n-1}}\right| <  \frac{n(n-1)}{2} q^{1/2}.
\]

Denoting by $B$ the number of bad  equivalence classes, we see that the contribution $\fG$ to~\eqref{eq: Idn via Jf}from good  equivalence classes
satisfies
\[
\left |\fG -\( \# \cI_{d,n}/q - B\)  \frac{q}{t_0\ldots t_{n-1}}  \right| < \( \# \cI_{d,n}/q - B\)  \frac{n(n-1)}{2} q^{1/2}, 
\]
which we write in a slightly weaker form
\begin{equation}
\label{eq:Good}
\left |\fG -\frac{1}{t_0\ldots t_{n-1}} \# \cI_{d,n}  \right| < \frac{n(n-1)}{2}  \# \cI_{d,n}  q^{-1/2}.
\end{equation}
Using the trivial bound $ \# \cJ_f(\vh,\vG) \le q$, we see that the contribution $\fB$ from bad  equivalence classes
satisfies
\begin{equation}
\label{eq:Bad}
\fB \le qB. 
\end{equation}

Using~\eqref{eq: Idn via Jf}, we now write 
\begin{align*}
 \left|\#  \cI_{d,n} (\vh, \vG)-    \frac{1}{t_0\ldots t_{n-1}} \# \cI_{d,n}  \right| 
&  = \left| \fG + \fB-    \frac{1}{t_0\ldots t_{n-1}} \# \cI_{d,n}  \right| \\
&  \le \left| \fG  -    \frac{1}{t_0\ldots t_{n-1}} \# \cI_{d,n}  \right| +  \fB . 
\end{align*}

Therefore, combining~\eqref{eq:Good} and~\eqref{eq:Bad}, we derive  
\begin{equation}
\label{eq: I and B}
 \left|\#  \cI_{d,n} (\vh, \vG)-  \frac{1}{t_0\ldots t_{n-1}} \# \cI_{d,n}   \right| \le \frac{n(n-1)}{2}  \# \cI_{d,n}  q^{-1/2} + qB.
\end{equation}

\subsection{Estimating the number of  bad equivalence classes and concluding the proof}

We see from Lemma~\ref{lem: Deg Intersect}  that for each pair $(i,j)$ with $0\le i < j \le n-1$,
the variety  $\cD_{d,n} \cap \cR_{i,j,n}$ is of degree at most 
\[
\deg\(\cD_{d,n} \cap \cR_{i,j,n}\) \le \deg\cD_{d,n}\cdot  \deg \cR_{i,j,n}
\le n (n-i) (n-j).
\] 
Therefore,  we see from  Lemma~\ref{lem: Gen S-Z_Lem} and Corollary~\ref{cor:discr res var} 
that there are at most $n(n-i)(n-j)q^{n-2}$ elements in  $\cD_{d,n} \cap \cR_{i,j,n}\cap \F_q[X]$. 
Summing over $0\le i < j \le n-1$ and recalling that each equivalence class contains $q$ polynomials, 
we derive
\begin{align*}
B &  \le  n q^{n-3}  \sum_{0\le i < j \le n-1}   (n-i) (n-j) \\
& =\frac{1}{2}  n q^{n-3}  \sum_{0\le i \ne j \le n-1}   (n-i) (n-j) \\
& = \frac{1}{2} n q^{n-3} \( \(\sum_{0\le i \le n-1}(n-i)  \)^2 -  \sum_{0\le i \le n-1}   (n-i)^2\)\\
& = \frac{1}{2} n q^{n-3} \( \frac{n^2(n+1)^2}{4}  - \frac{n(n+1)(2n+1)}{6} \)\\
& =   \frac{n^2(n^2-1)(3n+2)}{24} q^{n-3},
\end{align*}
which, after recalling~\eqref{eq: I and B}, implies the  result. 

\begin{remark}
In particular, if $n=2$ and 
\[
q >  \frac{2^2(2^2-1)(3\cdot 2+2)}{24} = 4
\] 
we get $B = 0$.
\end{remark}


\appendix 
\section{Counting irreducible polynomials with a given discriminant}
\label{app:Irr Poly Discr}

First we record the following (perhaps well-known) 
statement. 

\begin{lemma}
\label{lem: Abd Irred} Let $L_i(Z_1, \ldots, Z_m) \in \K[Z_1, \ldots, Z_m]$, $1\le i\le k$, be 
$k$ pairwise non-proprtional linear forms in $m$ variables over a field $\K$. Then for any $t \in \K^*$ 
the polynomial
\[
\prod_{i=1}^k L_i(Z_1, \ldots, Z_m) - t
\] 
is irreducible. 
\end{lemma} 

\begin{proof}  To see this, we consider a homogeneous polynomial 
\[
F(Z_1, \ldots, Z_m; T) = \prod_{i=1}^kL_i(Z_1, \ldots, Z_m) - tT^k,
\] which we treat 
as a polynomial in $T$ over the ring $\cR = \K[Z_1, \ldots, Z_m]$. By the Eisenstein criterion, the polynomial $F$ is irreducible
over $\cR$. 
On the other hand, it is easy to see that any factorisation of 
\[
\prod_{i=1}^kL_i(Z_1, \ldots, Z_m) - t = g(Z_1, \ldots, Z_m)  h(Z_1, \ldots, Z_m) 
\]
will induce a factorisation  $F = GH$ with 
\[
G= T^{\deg g}   g(Z_1/T, \ldots, Z_m/T) \quad \text{and}\qquad H= T^{\deg h}   h(Z_1/T, \ldots, Z_m/T).
\]
It remains to notice that since  $t \in \K^*$ both $G$ and $H$ are of positive degree with respect to $T$, which contradicts the already established irreducibility of $F$ over $\cR$. 
\end{proof}  

We now obtain an asymptotic formula for $\#\cI_{d,n}$. 

 \begin{theorem}
\label{thm:Irr Discr} Let $\F_q$ be of characteristic $p >  2$. 
For $d \in  \F_q^*$,  we have 
\[
\#\cI_{d,n}  =   \frac{2}{n} q^{n-1}  + O\(n^2 q^{n-3/2}\)
\]
if $\chi_2(d) = (-1)^{n-1}$, 
where $\chi_2$ is the quadratic character of $\F_q$,  and $\#\cI_{d,n} = 0$, otherwise. 
\end{theorem}

\begin{proof} It is easy to see that $\#\cI_{d,n}  = n^{-1} \# \cA_{d,n}$, where $\cA_{d,n}$ is the set of $\alpha \in \F_{q^n}$ 
which do not belong to any proper subfield of $\F_{q^n}$  and such that $d$ is the 
discriminant of the minimal polynomial $f_\alpha$ of $\alpha$.

Let $\sigma$ be the Frobenius automorphism, that is, $\sigma(\alpha) = \alpha^q$ for $\alpha \in \F_{q^n}$. 
Since the roots of $f_\alpha$ are 
\[
\sigma^0(\alpha) = \alpha, \sigma(\alpha) = \alpha^q , \ldots, \sigma^{n-1}(\alpha) = \alpha^{q^{n-1}},
\]
we see that 
\[
\discr f_\alpha = \prod_{0 \le i < j \le n-1} \(\sigma^i(\alpha) - \sigma^j(\alpha)\)^2.
\]

We now fix a basis $\omega_1, \ldots, \omega_n$ of $\F_{q^n}$ over $\F_q$ and consider the polynomial
\[
V(X_1, \ldots, X_n) =  \prod_{0 \le i < j \le n-1} \sum_{k=1}^n\(\sigma^i(\omega_k) - \sigma^j(\omega_k)\) X_k.
\]
Clearly if   $\alpha \in \F_{q^n}$   is written as $\alpha = x_1\omega_1+\ldots + x_n\omega_n$ then 
$\alpha \in \cA_{d,n}$ is equivalent to 
\begin{equation}
\label{eq: V2 = d}
V(x_1, \ldots, x_n)^2 = d. 
\end{equation}

Next we observe that 
\begin{align*}
\sigma (V(X_1, \ldots, X_n)) & =  \prod_{0 \le i < j \le n-1} \sum_{k=1}^n\sigma \(\sigma^i(\omega_k) - \sigma^j(\omega_k)\) X_k\\
& =  \prod_{0 \le i < j \le n-1} \sum_{k=1}^n \(\sigma^{i+1}(\omega_k) - \sigma^{j+1}(\omega_k)\) X_k, 
\end{align*}
and using that $\sigma^{n} = \sigma^0$, we derive 
\[
\sigma (V(X_1, \ldots, X_n))  = (-1)^{n-1} V(X_1, \ldots, X_n). 
\]
We now fix an arbitrary element $\gamma \ne 0$ in the algebraic closure $\overline{\F}_q$ of $\F_q$ with 
$\sigma(\gamma) =  (-1)^{n-1} \gamma$ (note that $\sigma(\gamma^2) =  \sigma(\gamma)^2 = \gamma^2$ thus in fact $\gamma^2\in \F_{q}$). 
Then for the polynomial $W(X_1, \ldots, X_n) = \gamma^{-1} V(X_1, \ldots, X_n)$ we have 
\[
\sigma (W(X_1, \ldots, X_n))  =  W(X_1, \ldots, X_n)
\] 
and thus $W(X_1, \ldots, X_n)$ is defined over $\F_q$.  

Next, we observe that~\eqref{eq: V2 = d} is equivalent to 
\begin{equation}
\label{eq: W2 = d}
W(x_1, \ldots, x_n)^2 = \gamma^{-2}d. 
\end{equation}
From the definition of $\gamma$ and the Euler criterion, we derive that 
\[
\chi_2(\gamma^2) = \(\gamma^2\)^{(q-1)/2} = \gamma^{q-1} = (-1)^{n-1}.
\]
Clearly, the equation~\eqref{eq: W2 = d}  has no solution in $x_1, \ldots, x_n \in \F_q$ unless 
\[
\chi_2(d) = \chi_2(\gamma^2)  = (-1)^{n-1},
\]
and splits into two equations 
\begin{equation}
\label{eq: W = pm t}
W(x_1, \ldots, x_n) = \pm t
\end{equation}
for some $t \in\F_q^*$. 

Recalling Lemma~\ref{lem: Abd Irred}, and using the asymptotic formula of  Lang and Weil~\cite{LaWe} 
to count solutions to each of the equations~\eqref {eq: W = pm t},  we conclude the proof. 
\end{proof}

\section*{Acknowledgements} 
 
During the preparation of this paper,  the authors were supported by  Australian Research Council Grants DP230100530 and FT250100208. The asymptotic 
formula and outline of the proof of Theorem~\ref{thm:Irr Discr} were provided by ChatGPT Pro~5.6, however, all details of the proof have been checked and written by the authors.

\end{document}